\documentclass[11pt]{amsart}
\usepackage[T1]{fontenc}

\usepackage{amsmath,amsthm,amsfonts,amssymb,latexsym,mathrsfs,graphicx}
\usepackage[utf8]{inputenc}

\usepackage{hyperref}
\usepackage{longtable}
\usepackage[capitalize, noabbrev]{cleveref}
\usepackage[all]{xy}

\usepackage{enumerate}
\usepackage[shortlabels]{enumitem}
\usepackage{color}
\usepackage{comment}
\usepackage[style=alphabetic, sorting=nyt, backend=biber]{biblatex}
\newcommand{\Z}{\mathbb Z}
\newcommand{\Q}{\mathbb Q}
\newcommand{\C}{\mathbb C}
\newcommand{\R}{\mathbb R}

\def \U{\mathcal{U}}

\newcommand{\Irr}{\operatorname{Irr}}
\newcommand{\Gal}{\operatorname{Gal}}

\newtheorem{theorem}{Theorem}[section]
\newtheorem{corollary}[theorem]{Corollary}
\newtheorem{lemma}[theorem]{Lemma}
\newtheorem{proposition}[theorem]{Proposition}
\theoremstyle{definition}
\newtheorem{definition}[theorem]{Definition}

\newtheorem{example}[theorem]{Example}
\newtheorem{examples}[theorem]{Examples}
\newtheorem{remark}[theorem]{Remark}

\newtheorem{maintheorem}{Theorem}

\def\irr#1{{\rm Irr}(#1)}
\def\cent#1#2{{\bf C}_{#1}(#2)}

\def\zent#1{{\bf Z}(#1)}

\newcommand{\F}{{\mathbb F}}
\newcommand{\K}{{\mathbb K}}

\def\irr#1{{\rm Irr}(#1)}

\def\cent#1#2{{\bf C}_{#1}(#2)}

\def\norm#1#2{{\bf N}_{#1}(#2)}

\def\zent#1{{\bf Z}(#1)}

\def\Z{{\mathbb Z}}
\def\C{{\mathbb C}}
\def\Q{{\mathbb Q}}
\def\irr#1{{\rm Irr}(#1)}

\def\cent#1#2{{\bf C}_{#1}(#2)}

\def\zent#1{{\bf Z}(#1)}

\def\norm#1#2{{\bf N}_{#1}(#2)}

\def \mod#1{\, {\rm mod} \, #1 \, }

\mathchardef\coso="2023

\usepackage[normalem]{ulem}
\usepackage{cancel}
\usepackage[capitalize, noabbrev]{cleveref}

\newcommand{\GEN}[1]{\left\langle #1 \right\rangle}

\newcommand{\mcR}{\mathcal{R}}
\newcommand{\mcS}{\mathcal{S}}

\newcommand{\qand}{\quad \text{and} \quad}

\begin{document}
	
	\title{Uniformly semi-rational groups}
	
\begin{abstract}
We introduce and study some families of groups whose irreducible characters take values on quadratic extensions of the rationals.
We focus mostly on a generalization of inverse semi-rational groups, which we call uniformly semi-rational groups. 
Moreover, we associate to every finite group two invariants, called rationality and semi-rationality of the group. They measure respectively how far a group is from being rational and how much uniformly rational it is. We determine the possible values that these invariants may take for finite nilpotent groups. We also  classify the fields that can occur as the field generated by the character values of a finite nilpotent group. 
\end{abstract}

	\author[]{Ángel del Río}
	\address{Ángel del Río, Departamento de Matemáticas,\newline
		Universidad de Murcia, Campus de Espinardo, 30100
		Murcia, Spain.}
	\email{adelrio@um.es}
	
	\author[]{Marco Vergani}
	\address{Marco Vergani, Dipartimento di Matematica e Informatica U. Dini,\newline
		Universit\`a degli Studi di Firenze, viale Morgagni 67/a,
		50134 Firenze, Italy.}
	\email{marco.vergani@unifi.it}
	
\thanks{The first author is partially supported by Grant PID2020-113206GB-I00 funded by MICIU/AEI/10.13039/501100011033 and by Grant 22004/PI/22 of Fundación Séneca de la Región de Murcia. The second author is partially supported by INdAM-GNSAGA. This research is also funded by the European Union-Next Generation EU, Missione 4 Componente 1, CUP B53D23009410006, PRIN 2022 2022PSTWLB - Group Theory and Applications.}

	\maketitle
	
\section{Introduction}\label{SecIntro}

All throughout this paper $G$ is a finite group. 
Recall that $G$ is said to be \emph{rational} or \emph{rational-valued} if every character of $G$ takes values on the field $\Q$ of rational numbers. 
It is well-known that $G$ is rational if and only if for every $g\in G$, each generator of $\GEN{g}$ is conjugate to $g$ in $G$ \cite[Theorem~V.13.7]{huppertgroupsI}. 
The study of rational groups is a classical topic in representation theory (see e.g. \cite{Kletzing1984,FeitSeitz1989,Thompson2008,IsaacNavarro2012}). 
Several generalizations allowing the characters taking values on quadratic extensions of the rationals have been considered during the latter years. This was suggested by G. Navarro \cite{oddquad} and eventually led to the notions of semi-rational groups, inverse semi-rational groups and quadratic rational groups \cite{CD,tentquadraticrat}.  
Inverse semi-rational groups are also known as cut groups, where cut is an acronym of ``\textbf{c}entral \textbf{u}nits are \textbf{t}rivial'', as they are precisely the finite groups for which all the central units of their integral group ring are trivial \cite{RitterSehgal1990, Bachle2018}.
The aim of this paper is to introduce a generalization of the concept of inverse semi-rational groups, that we call uniformly semi-rational, and associate to each finite group $G$ two subsets of the group of units of integers modulo the exponent of $G$, which we call the rationality and semi-rationality of $G$. 
The rationality measures ``how much rational'' a group is, while the semi-rationality quantifies ``how much uniformly semi-rational'' it is.
Associated to these sets, several natural problems arise which we address for nilpotent groups.
On the other hand, other group theoretical notions enter naturally into the picture when studying uniformly semi-rational groups. We show some connections and characterizations of these notions.
To introduce and motivate the rationality, the semi-rationality and the mentioned notions, we need some notation. 

Let $\Irr(G)$ denote the set of ordinary irreducible characters of $G$ and consider the fields:
\begin{eqnarray*}
\Q(G) &=& \Q(\chi(g) : \chi\in \Irr(G), g\in G\}, \\
\Q(\chi) &=& \Q(\chi(g) : g\in G), \text{ for each } \chi\in\Irr(G), \text{ and }\\
\Q(g) &=& \Q(\chi(g) : \chi\in \Irr(G)), \text{ for each } g\in G. 
\end{eqnarray*}

Let $n$ be the exponent of $G$, let $\U_n$ denote the group of units of $\Z/n\Z$ and let $\Q_n$ denote the $n$-th cyclotomic extension of $\Q$.
The \emph{rationality} is the following subgroup of $\U_n$:
$$\mcR_G = \{r\in \U_n : g\text{ is conjugate to } g^r \text{ in } G, \text{ for every } g\in G\}.$$

Following \cite{CD} and \cite{tentquadraticrat}, we say that $G$ is  
\begin{itemize}
\item \emph{semi-rational} if $[\Q(g):\Q]\le 2$ for every $g\in G$;
\item \emph{inverse semi-rational} if for every $g\in G$, every generator of $\GEN{g}$ is conjugate either to $g$ or to $g^{-1}$ in $G$;
\end{itemize}
Inspired by these definitions, we propose to say that $G$ is 
\begin{itemize}
\item \emph{quadratic} if $[\Q(G):\Q]\le 2$;
\item \emph{character quadratic} if  $[\Q(\chi):\Q]\le 2$ for every $\chi\in \Irr(G)$. In \cite{tentquadraticrat}, it is named quadratic rational.
\item \emph{quadratic valued} if $[\Q(\chi(g)):\Q]\le 2$ for every $g\in G$ and $\chi\in \Irr(G)$.
\end{itemize}

It turns out that $G$ is semi-rational if for every $g\in G$ there is $r_g\in \U_n$, depending on $g$, such that every generator of $\GEN{g}$ is conjugate to $g$ or $ g^{r_g}$ in $G$.
Observe that $G$ is inverse semi-rational when $r_g$ can be taken as $-1$ for every $g\in G$.
This suggest the following generalizations of inverse semi-rational groups. 
We say that $G$ is 
\begin{itemize}
\item \emph{$r$-semi-rational}\footnote{Observe that in \cite[Problem~2]{CD}, the notion of ``$k$-semi-rational'' is defined with a different meaning. We would not make use of that definition.} with $r\in \U_n$, if for every $g\in G$ each generator of $\GEN{g}$ is conjugate to $g$ or $g^{r}$ in $G$.  
\item \emph{uniformly semi-rational}, abbreviated \emph{USR}, if it is $r$-semi-rational for some $r\in \U_n$.
\item \emph{quadratic conjugated} if $g$ is conjugate to $g^{s^2}$ in $G$,  for every $g\in G$ and $s\in \U_n$.
\end{itemize}

\begin{table}[ht]
    \centering
    \begin{tabular}{|l|c|c|}
\hline
Type of groups & $|G|<512 $ & $|G|=512$\\
\hline
Quadratic conjugated & 87.31 $\%$ & 99.73$\%$\\
Quadratic valued & 76.50 $\%$ & 98.69$\%$ \\
Character quadratic & 59.97 $\%$ & 90.07$\%$\\
Semi-rational & 61.15 $\%$  & 94.96$\%$\\
Semi-rational and character quadratic  & 55.45 $\%$ & 88.96$\%$\\
USR & 51.78 $\%$ & 88.35$\%$\\
Quadratic &43.20 $\%$ &  85.73$\%$\\
Inverse semi-rational & 45.96 $\%$ &  87.00$\%$\\
Inverse semi-rational and quadratic & 42.23 $\%$ & 85.69$\%$\\
Rational & 1.17 $\%$ & 0.55$\%$\\\hline
\end{tabular}
    \smallskip
    \caption{ The first column of this table contains the percentage of isomorphism classes of groups of each type, among 92\;804 isomorphism classes of groups of order less than 512. The second one contains the percentages in a random sample of 100\;000 groups of order 512. The calculations were performed using the GAP library of small groups.
    \label{TableFamilies}}
\end{table}

\Cref{TableFamilies} indicates that, while rational groups are quite sparse, the families introduced above are more frequent, specially among $2$-groups.
\Cref{Implications} shows the logical connections between these concepts. Most of the implications are obvious; the others will be proved along the paper in the results indicated in the figure. See also \Cref{NoConverses}.

	\begin{figure}[ht!]
		$$\xymatrix{
			&	\text{Rational} \ar@{->}[ld] \ar@{->}[rd] &\\
			\text{Inverse semi-rational} \ar@{->}[rd] & & \text{Quadratic} \ar@{->}[ld]^{\txt{ \quad Prop. \ref{QsRCar1}}} \\
			&\text{USR} \ar@{->}[ld]\ar@{->}[rd]^{\txt{ \quad Cor. \ref{USRQR}}}& \\
			\text{Semi-rational} \ar@{->}[rd] &  & \text{Character quadratic} \ar@{->}[ld]\\
			& \text{Quadratic valued}\ar@{->}[d]^{\txt{Cor. \ref{QVImpliesQC}}} & \\
			& \text{Quadratic conjugated} & 
		}$$
		\caption{\label{Implications} }
	\end{figure}

Our first result is the following characterization of USR groups.

\begin{maintheorem}\label{USR}
The following are equivalent for a finite group $G$ of exponent $n$. 
\begin{enumerate}
	\item\label{USR-USR} $G$ is USR.
    \item\label{USR-Irr} $G$ is quadratic conjugated and $\Q_n$ is a cyclic extension of a field $K$ such that $\Q(\chi)\cap K=\Q$ for every $\chi\in\Irr(G)$.
   \item\label{USR-Qg} $G$ is quadratic conjugated and $\Q_n$ is a cyclic extension of a field $K$ such that $\Q(g)\cap K=\Q$ for every $g\in G$.
\end{enumerate}
\end{maintheorem}

The \emph{semi-rationality} of $G$  is the following subset of $\U_G$:
    $$\mcS_G = \{r\in \U_G : G \text{ is } r\text{-semi-rational}\}.$$
Clearly the following are equivalent: (1) $G$ is rational, (2) $1\in \mcS_G$, (3) $\mcR_G=\U_G$, (4) $\mcS_G=\U_G$.
Moreover, $G$ is USR if and only if $\mcS_G\ne\emptyset$ and, in that case, $\mcS_G$ is a coset of $\U_G$ modulo $\mcR_G$ (see \Cref{SRCoset}).

Our next result describes the fields which occur as $\Q(G)$ for a nilpotent group $G$ of given exponent $n$, and the subsets of $\U_n$ that arise as rationality or semi-rationality of such groups. 

\begin{maintheorem}\label{Nilpotent}
Let $n$ be positive integer and let $n'$ be the greatest square-free divisor of $n$. Then
\begin{enumerate}
    \item The fields which occur as $\Q(G)$ for a finite nilpotent group $G$ of exponent $n$ are the subextensions of $\Q_n/\Q_{n'}$.
    \item The subsets that occur as the rationality of a finite nilpotent group of exponent $n$ are the subgroups of $\{r\in \U_n : r\equiv 1 \mod n'\}$.
    \item If $S$ is a subset of $\U_n$, then there is a finite nilpotent USR group of exponent $n$ with semi-rationality $S$ if and only if $n$ is not divisible by any prime greater than $3$, $S$ is either $\U_n$ or a coset of $\U_n$ modulo a subgroup $R$ such that $\{r^2:r\in \U_n\}\subseteq R\ne S$, and if $3$ divides $n$, then $S=\{x\in \U_n : x\equiv -1 \mod 3\}$.
\end{enumerate}
\end{maintheorem}

It would be interesting to obtain versions of \Cref{Nilpotent} for other families of groups, i.e. we suggest to investigate the following problems: For a given family $\mathcal C$ of groups classify
\begin{itemize}
\item[(1)] the fields which occur as $\Q(G)$ for a group $G$ in $\mathcal C$.
\item[(2)] the subfields of $\Q_n$ which occur as $\Q(G)$ for a group $G$ in $\mathcal C$ of exponent $n$.
\item[(2')] the subgroups of $\U_n$ which occur as the rationality of a group  in $\mathcal C$ of exponent $n$.
\item[(3)] the cosets of $\U_n$ which occur as the semi-rationality of a group in $\mathcal C$ of exponent $n$.
\end{itemize}

Clearly, a solution of (2) for every $n$ yields a solution for (1).
It turns out that the natural isomorphism $\U_n\to \Gal(\Q_n/\Q)$ maps $\mcR_G$ to $\Gal(\Q_n/\Q(G))$. This shows that the problems (2) and (2') are equivalent.

The paper is organized as follows: In \Cref{SecNotation} we introduce the basic notation. In \Cref{SecBasic}  we establish some basic facts and obtain several characterizations of the families of groups defined above. \Cref{USR} is a direct consequence of  \Cref{Stillness} and \Cref{r-semirational}, which furthermore provides a way to compute the semi-rationality of a group from its character table.
\Cref{SecNilpotent} is dedicated to study the concepts introduced above for nilpotent groups. The three statements of \Cref{Nilpotent} rephrase \Cref{QGNilp}, \Cref{RatNilp} and \Cref{USRNilpotent}.

\section{Notation}\label{SecNotation}
	
Fix a positive integer $n$. 
Let $\pi(n)$ denote the set of primes dividing $n$, let $\U_n$ denote the group of units of $\Z/n\Z$, let $\zeta_n$ denote a complex primitive $n$-th root of unity and set $\Q_n=\Q(\zeta_n)$. 
We abuse the notation by expressing an element $r$ of $\Z/n\Z$ by any of its representatives and writing $\zeta_n^r$, and more generally $g^r$ for a group element $g$ of order dividing $n$, with the obvious meaning.
If $r\in \U_n$, then $\sigma_r$ denotes the automorphism of $\Q_n$ that maps $\zeta_n$ to $\zeta_n^r$. 
Then $r\mapsto \sigma_r$ is an isomorphism $\U_n\rightarrow \Gal(\Q_n/\Q)$. 
Furthermore, $\U_n^2$ denotes the set of squares of elements in $\U_n$.
An \emph{admissible coset} of $\U_n$ is a coset $S$ of $\U_n$ modulo a subgroup $V$ of $\U_n$ such that either $S=V=\U_n$ or $S\ne V$ and $\U_n^2\subseteq V$. 
This rather artificial notion will be useful for our results about the semi-rationality of a group. 

We use standard notation $C_n$ and $D_n$ for cyclic and dihedral groups, respectively, of order $n$. Moreover, $\GEN{a}_n$ represents a cyclic group of order $n$ generated by $a$.
Other groups which appear in this paper are the semi-dihedral groups, for $k\ge 3$:
$$D_{2^{k+1}}^-=\GEN{a}_{2^k}\rtimes \GEN{b}_2, \text{  with }a^b=a^{-1+2^{k-1}}.$$
Furthermore \texttt{SmallGroup(n,m)} represents the $m$-th group of order $n$ in the \texttt{GAP} library of small groups \cite{GAP4}.

Throughout this paper every group is finite and $G$ is a fixed group.
Let $\sim$ denote the conjugacy relation in $G$,  and for each $g\in G$, let $g^G$ denote the conjugacy class of $g$ in $G$. If $n=\exp(G)$, the exponent of $G$, then we denote, $\pi(G)=\pi(n)$, $\U_G=\U_n$ and $\Q_G=\Q_n$, and $\sigma_G:\U_G\to \Gal(\Q_G/\Q)$ is the isomorphism that maps $r$ to $\sigma_r$.  
    When the group $G$ is clear from the context, we write $\sigma$ instead of $\sigma_G$.
 
 If $H$ is a subgroup of $G$, then the normalizer and centralizer of $G$ in $H$ are denoted  $\norm G H$ and $\cent G H$, respectively.
We also use $\cent G g$ for the centralizer of an element $g$ in $G$.

	If $g\in G$, then $T(g):\Irr(G)\to \C$ is the map given by $T(g)(\chi)=\chi(g)$.  Observe that $T(g)=T(h)$ if and only if $g\sim h$. Let $T(G)$ denote the set of maps $T(g)$ with $g\in G$.

\begin{examples}\label{NoConverses}
The following list of examples shows that all the implications in \Cref{Implications} are proper. 
\smallskip

\begin{itemize}
\item $C_3$ and $C_4$ are inverse semi-rational and quadratic but not rational.
\item $D_{16}$ is quadratic but not inverse semi-rational.
\item $\GEN{a}_8\rtimes \GEN{b}_4$ with $b^a=a^3$, is inverse semi-rational but not quadratic. 
\item $\GEN{a}_8\rtimes \GEN{b}_4$ with $b^a=a^{-1}$, is USR but neither inverse semi-rational nor quadratic. 
\item \texttt{SmallGroup(128,417)}$=\GEN{a}_8\rtimes (\GEN{b}_8\rtimes \GEN{c}_2)$, with $a^b=a^{-1}$, $a^c=a^3$, $b^c=b^{-1}$, is semi-rational and character quadratic but not USR.
\item \texttt{SmallGroup(32,15)}=$\GEN{a,b \mid a^8=1, b^4=a^4, a^b=a^3}$ is semi-rational but not character quadratic.
\item \texttt{SmallGroup(32,9)}
$=(\GEN{a}_8\times \GEN{b}_2)\rtimes \GEN{c}_2$, with $a^c=ab$, $b^c=a^4b$ is character quadratic but not semi-rational.
\item \texttt{SmallGroup(64,22)} is quadratic valued but neither semi-rational nor character quadratic.
\item $C_8$ is quadratic conjugated but not quadratic valued.
\end{itemize}	
\end{examples}

\begin{definition}
Let $g\in G$. Then we denote
	$$\mcR_g=\{r\in \U_G : g\sim g^r\}.$$ 
The \emph{rationality} of $G$ is the following subgroup of $\U_G$:
	$$\mcR_G=\bigcap_{g\in G} \mcR_g = \{r\in \U_G : g\sim g^r \text{ for all } g\in G\}.$$
	\end{definition}
Recall that $g$ is said to be \emph{rational} in $G$ if $g$ is conjugate  in $G$ to each generator of $\GEN{g}$, equivalently if $\mcR_g=\U_G$.
	It is easy to see that $g$ is rational in $G$ if and only if $[\norm G {\GEN{g}}:\cent G g]=\varphi(n)$, where $\varphi$ denotes the Euler totient function, that is equivalent to say that $\Q(g)=\Q$.
	Hence $G$ is rational if and only if every element of $G$ is rational in $G$ if and only if $\mcR_G=\U_G$.
    
Observe that 
\begin{equation}\label{RGGal}
    \mcR_G=\sigma^{-1}(\Gal(\Q_G/\Q(G)))\quad
    \text{and hence}\quad [\U_G : \mcR_G]=[\Q(G):\Q].
\end{equation}
Indeed, let $r\in \U_G$. Then $\sigma_r(\chi(g))=\chi(g^r)$ for every $\chi\in G$. Since $g\sim g^r$ if and only if $\chi(g)=\chi(g^r)$ for every $\chi\in \irr G$, it follows that $r\in \mcR_g$ if and only if $\chi(g)=\sigma_r(\chi(g))$ for every  $\chi\in\Irr(G)$, equivalently if $\sigma_r\in \Gal(\Q_G/\Q(g))$.
Note that \eqref{RGGal} is equivalent to the following 
$$\Q(G)=\{a\in \Q_G : \sigma_r(a)=a \text{ for every } r\in \mcR_G\}.$$

	\begin{definition} 
	    Let $r\in \U_G$. We say that $g$ is \emph{$r$-semi-rational} if every generator of $\GEN{g}$ is conjugate to either $g$ or $g^r$, i.e. if $\U_G=\mcR_g\cup r\mcR_g$. We denote
	$$\mcS_g=\{r\in \U_G : g \text{ is } r-\text{semi-rational}\}.$$
	The \emph{semi-rationality} of $G$ is the following subset of $\U_G$: 
	$$\mcS_G = \bigcap_{g\in G} \mcS_g.$$
	\end{definition}
    
The element $g$ is said to be \emph{semi-rational} in $G$ if it is $r$-semi-rational in $G$ for some $r\in \U_G$.

\section{Basic properties and characterizations}\label{SecBasic}
In this section we focus on characterizing USR groups by several means and, in particular, on proving \Cref{USR}. 
On the way, we also obtain basic properties of the semi-rationality of a group and connections between the different families of groups introduced before. 

\subsection{Basic properties}
In this subsection, we collect some basic facts about the families of groups introduced in \Cref{SecIntro}. We start with some characterizations of semi-rational elements, which follow from \cite[Lemma~5 (1)]{CD} and \cite[Lemma~1]{tentquadraticrat}, and elementary arguments.
 \begin{lemma}\label{CharSemirat}
     The following are equivalent.
    \begin{enumerate}
        \item $g$ is semi-rational in $G$.
        \item $[\norm{G}{\GEN{g}}:\cent G g]\geq \frac{\varphi(|g|)}{2}$.
        \item  $[\Q(g):\Q]\leq 2$. 
        \item $[\U_G:\mcR_g]\leq 2$.
    \end{enumerate}
 \end{lemma}

 Clearly $G$ is USR if and only if $\mcS_G\ne \emptyset$. Moreover
	$$\mcS_g = \begin{cases} \U_G, & \text{if } g \text{ is rational in } G;\\
	\emptyset, & \text{if } g \text{ is not semi-rational in } G; \\
	\U_G\setminus \mcR_g\ne \emptyset, & \text{otherwise}.\end{cases}$$
Combining this with \Cref{CharSemirat} it follows that if $g$ is semi-rational in $G$, then $\mcS_g$ is a coset of $\mcR_g$ in $\U_G$ and $\U_G^2\subseteq \mcR_g$. As an intersection of cosets of a group is either empty or a coset modulo the intersection of the corresponding subgroups we obtain the following.

\begin{proposition}\label{SRCoset}
If $G$ is USR then $\mcS_G$ is an admissible coset of $\U_n$ modulo $\mcR_G$.
\end{proposition}

\begin{proposition}\label{CarQC}
The following are equivalent:
\begin{enumerate}
	\item $G$ is quadratic conjugated.
    \item $\Q(G)$ is a compositum of extensions of $\Q$ of degree at most $2$.
    \item $\Gal(\Q(G)/\Q)$ is elementary abelian $2$-group.
	\item $\Gal(\Q(\chi)/\Q)$ is elementary abelian $2$-group for every $\chi\in\Irr(G)$.
	\item $\Gal(\Q(g)/\Q)$ is elementary abelian $2$-group for every $g\in G$.
   \item $\Gal(\Q(\chi(g))/\Q)$ is elementary abelian $2$-group for every $g\in G$ and $\chi\in\irr G$.
		\end{enumerate}
	\end{proposition}
	\begin{proof}Recall that a Galois extension $K/k$ is said to be of exponent $n$ if $\tau^n=1$ for every $\tau\in\Gal(K/k)$ \cite[p.~293]{Lang}. By Kummer's Theorem \cite[Theorem~8.2]{Lang}, $K/k$ has exponent $2$ if and only if $K$ is the compositum of extensions of degree at most $2$ over $k$. Moreover, equation~\eqref{RGGal} implies that $\U_G/\mcR_G\cong \Gal(\Q_G/\Q)/\Gal(\Q_G/\Q(G))\cong \Gal(\Q(G)/\Q)$. This shows that the first three conditions are equivalent. The same argument gives the equivalence with the remaining conditions.
	\end{proof}

\begin{corollary}\label{QVImpliesQC}
    If $G$ is quadratic valued, then it is quadratic conjugated.   
\end{corollary}

\subsection{Stillness}

In this subsection we introduce an essential notion for our characterization of USR groups. 
We use standard notation $\zent R$ for the center of a ring $R$, and $RG$ for the group ring of $G$ with coefficients in $A$.

Recall that every semisimple module is the direct sum of its homogeneous components. We say that a semisimple module is \emph{homogeneous} if it has exactly one non-zero homogeneous component. Two homogeneous modules are equivalent if they have isomorphic simple direct summands. So the number of isomorphism classes of simple modules of a ring equals the number of equivalence classes of semisimple homogenous modules. Furthermore, if $A$ is $K$-algebra, $F$ is a field extension of $K$, $H$ is a homogeneous $A$-module and $S$ a simple direct summand of $H$, then $F\otimes_K H$ and $F\otimes_K S$ have the same number of homogeneous components.

\begin{definition}
Let $K$ be a subfield of $\C$. We say that $G$ is $K$-\emph{still} if for every homogeneous $\Q G$-module $H$, the $K G$-module $K\otimes_\Q H$ is homogeneous.
 \end{definition}

Let $K$ be a field and let $A$ be a semisimple $K$-algebra. 
Then $A=\oplus_{i=1}^n A_i$ with each $A_i$ a simple $K$-algebra. 
Then $A_1,\dots,A_n$ is a set of representatives of the equivalence classes of homogeneous $A$-modules.
Let now $F$ be a field extension of $K$ and $B=F\otimes_K A$. 
Then $F\otimes_K A_i=\oplus_{j=1}^{k_i} B_{i,j}$ with each $B_{i,j}$ a simple algebra and hence $B=\oplus_{i=1}^n \oplus_{j=1}^{k_i} B_{i,j}$. Thus the $B_{i,j}$ form a set of representatives of the equivalence classes of homogeneous $B$-modules. Therefore, the number of equivalence classes of homogeneous $B$-modules is $\sum_{i=1}^n k_i$. This implies that if $K$ is a subfield of $\C$, then $G$ is $K$-still if and only if $\Q G$ and $KG$ have the same number of equivalence classes of homogeneous modules, equivalently the number of $\Q$-characters equals the number of $K$-characters.

We denote
	$$\U_K=\U_K(G)=\sigma^{-1}(\Gal(\Q_G/\Q_G\cap K))=\{r\in \U_G :\; \sigma_r(x)=x \text{ for every } x\in K\cap \Q_G\}.$$
Recall that two elements $g$ and $h$ in $G$ are \emph{$K$-conjugate} in $G$, denoted $g\sim_K h$, if $h\sim g^r$ for some $r\in \U_K$ (see \cite[p.~306]{CurtisReiner}).
Then $g^G_K$ denotes the set formed by the elements of $G$ which are $K$-conjugate to $g$, i.e.     
    $$g^G_K=\bigcup_{r\in \U_K} (g^r)^G.$$ 
If $\chi\in \irr G$, then we set
 $$\chi_K=\{\chi^\sigma :\; \sigma\in \U_K\}.$$
	
\begin{proposition}\label{Stillness}
The following are equivalent for a subfield $K$ of $\C$:
\begin{enumerate}
	\item\label{SC} $G$ is $K$-still.
    \item \label{SCC} $g^G_K=g^G_\Q$ for every $g\in G$.
    \item \label{SCChi}  $\chi_K=\chi_\Q$ for every $\chi\in \irr G$.
	\item\label{SCIrr} For every $\sigma\in \Gal(\Q_G/\Q)$ and $\chi\in\Irr(G)$ there is $\tau_\chi\in \Gal(\Q_G/\Q_G\cap K)$ such that $\chi^\sigma=\chi^{\tau_\chi}$.
   	\item\label{SCg} For every $\sigma\in \Gal(\Q_G/\Q)$ and $g\in G$ there is $\rho_g\in \Gal(\Q_G/\Q_G\cap K)$ such that $\chi^\sigma(g)=\chi^{\rho_g}(g)$ for every $\chi\in\Irr(G)$.
	\item\label{SCIrrRes} The restriction map $\Gal(\Q_G/\Q_G\cap K)\rightarrow \Gal(\Q(\chi)/\Q)$ is surjective for every $\chi\in \Irr(G)$.
	\item\label{SCgRes} The restriction map $\Gal(\Q_G/\Q_G\cap K)\rightarrow \Gal(\Q(g)/\Q)$ is surjective for every $g\in G$.
   	\item\label{SCIrrInt} $\Q(\chi)\cap K=\Q$ for every $\chi\in\Irr(G)$.
   \item\label{SCgInt} $\Q(g)\cap K=\Q$ for every $g\in G$.
\end{enumerate}
\end{proposition}
	
\begin{proof}
By replacing $K$ with $\Q_G\cap K$ we may assume, without loss of generality, that $K\subseteq \Q_G$.
The equivalence between \eqref{SCChi}, \eqref{SCIrr} and \eqref{SCIrrRes} is obvious, as so is the equivalence between \eqref{SCg} and \eqref{SCgRes}.
 
By the Witt-Berman Theorem \cite[Theorem~42.8] {CurtisReiner}, the number of $K$-conjugacy classes of $G$ coincides with the number of isomorphism classes of irreducible $K G$-modules, which is the same as the number of equivalence classes of homogeneous $KG$-modules, and therefore $G$ is $K$-still if and only if $g^G_K=g^G_\Q$ for every $g\in G$. Hence \eqref{SC} and \eqref{SCC} are equivalent.
        
Condition \eqref{SCIrr} holds if and only if $\Gal(\Q_G/\Q)=\Gal(\Q_G/K)\Gal(\Q_G/\Q(\chi))$ for every $\chi\in \Gal(\Q_G/\Q)$. Since the Galois correspondent of $\Gal(\Q_G/K)\Gal(\Q_G/\Q(\chi))$ is $K\cap \Q(\chi)$, it follows that condition \eqref{SCIrr} and \eqref{SCIrrInt} are equivalent.

		The equivalence between \eqref{SCg} and \eqref{SCgInt} follows from the same argument. 
		
In order to prove that \eqref{SCC} implies \eqref{SCIrr} and \eqref{SCg}, suppose that $g^G_\Q=g^G_K$ for every $g\in G$.
Let $\sigma\in \Gal(\Q_G/\Q)$, $\chi\in\Irr(G)$ and $g\in G$. 
Then $\sigma=\sigma_j$ for some $j\in\U_G$. 
By assumption $g^j\in g^G_\Q=g^G_K$, i.e. $g^j\sim g^r$ for some $r\in\U_K$. 
Therefore 
\begin{eqnarray*}
\sum_{\tau\in \Gal(\Q_G/K)} \chi^{\sigma\tau}(g) &=& 
		\sum_{\tau\in \Gal(\Q_G/K)} \chi^\tau(g^j) = 
		\sum_{\tau\in \Gal(\Q_G/K)} \chi^{\tau}(g^r) \\
  &=& \sum_{\tau\in \Gal(\Q_G/K)}  \chi^{\sigma_r\tau}(g) =
		\sum_{\tau\in \Gal(\Q_G/K)} \chi^\tau(g).
\end{eqnarray*}
This shows that
$$\sum_{\tau\in\Gal(\Q_G/K)} \chi^\tau=\sum_{\tau\in\Gal(\Q_G/K)} \chi^{\sigma\tau} \qand
\sum_{\tau\in\Gal(\Q_G/K)} T(g)^\tau=\sum_{\tau\in\Gal(\Q_G/K)} T(g)^{\sigma\tau}.$$ 
As the sets $\Irr(G)$ and $T(G)$ are linearly independent over $\C$ and invariant under the action of $\Gal(\Q_G/\Q)$, it follows that there are $\tau_\chi$ and $\rho_g$ in $\Gal(\Q_G/K)$ such that $\chi^\sigma=\chi^{\tau_{\chi}}$ and $T(g)^{\sigma}=T(g)^{\rho_g}$.

We now prove that \eqref{SCg} implies \eqref{SCC}. Assume that \eqref{SCg} holds and let $g\in G$ and $j\in \U_G$. Then there is $r_g\in \U_K$ such that $\chi(g^j)=\chi^{\sigma_j}(g)=\chi^{\sigma_{r_g}}(g)=\chi(g^{r_g})$ for every $\chi\in\Irr(G)$. This implies that $g^j\sim g^{r_g}$. Thus $g^G_\Q=g^G_K$.
		
Finally, we prove that \eqref{SCIrrInt} implies \eqref{SCC}. 
Suppose that $\Q(\chi)\cap K=\Q$ for every $\chi\in \Irr(G)$. Let $\chi \in \Irr(G)$. Then $[\Q(\chi):\Q]=[K(\chi):K]$. Let $A$ be the unique simple component of $\Q G$ with $\chi(A)\ne 0$. By items (4) and (5) in \cite[Theorem~3.3.1]{JespersdelRioGRG1}, we get that $\zent A=\Q(\chi)$ and $K\otimes_{\Q} A$ is a central simple $K(\chi)$-algebra. This implies that $\Q G$ and $K G$ have the same number of simple components. By the Witt-Berman Theorem, the number of simple components of $K G$ is the number of $K$-classes of $G$. Thus the number of $\Q$-classes and $K$-conjugacy classes of $G$ coincides and hence $g^G_\Q=g^G_K$ for every $g\in G$.
	\end{proof}

	\subsection{Uniformly semi-rational groups}

 In this subsection we obtain a characterization of USR groups. To that aim, we introduce the following subfields of $\Q_G$:  
    $$\K_r=\{x\in \Q_G :\; \sigma_r(x)=x\}, \quad \text{for } r\in \U_G.$$ 
    
\begin{proposition}\label{r-semirational}
Let $r\in \U_G$. Then the following are equivalent.
\begin{enumerate}
    \item $G$ is $r$-semi-rational. 
    \item $G$ is quadratic valued and $\K_r$-still.
    \item $G$ is quadratic conjugated and $\K_r$-still.
\end{enumerate}
\end{proposition}
	
\begin{proof}
We start by proving (1) implies (2). Suppose that $G$ is $r$-semi-rational. Thus $G$ is quadratic valued. Moreover, for every $g\in G$ we have $g^G_{\K_r}\subseteq g^G_\Q=g^G\cup (g^r)^G\subseteq g^G_{\K_r}$. Therefore, by \Cref{Stillness}, $G$ is $\K_r$-still.

(2) implies (3) is a consequence of \Cref{QVImpliesQC} .
  
To prove (3) implies (1), we assume that $G$ is quadratic conjugated and $\K_r$-still. Let $g\in G$. We have to prove that $g$ is $r$-semi-rational in $G$. On the one hand, $\Gal(\Q(g)/\Q)$ is elementary abelian $2$-group, by \Cref{CarQC}; on the other hand, the restriction map $\GEN{\sigma_r}=\Gal(\Q_G/\K_r)\to \Gal(\Q(g)/\Q)$ is surjective, by \Cref{Stillness}. Then $[\Q(g):\Q]\leq 2$, and hence $g$ is semi-rational in $G$ by \Cref{CharSemirat}. If $g$ is rational in $G$, then it is $r$-semi-rational in $G$. Otherwise $[\Q(g):\Q]=2$ and, by \Cref{Stillness}, $\Q(g)\cap \K_r=\Q$. Therefore $\Q(g)$ is not contained in $\K_r$ and hence $\chi(g^r)=\sigma_r(\chi(g))\neq\chi(g)$ for some $\chi\in \irr G$. Thus $g\not \sim g^r$, proving that $g$ is $r$-semi-rational in $G$, as desired.
\end{proof}

\begin{corollary}\label{r2Semirat}
    Let $r\in \U_G$ with $r^2=1$. Then $G$ is $r$-semi-rational if and only if $G$ is $\K_r$-still. 
\end{corollary}
\begin{proof}
    The necessary condition is consequence of \Cref{r-semirational}. 
    To prove the sufficient condition, we assume that $r^2=1$ and $G$ is $\K_r$-still. Then the order of $\Gal(\Q_G/\K_r)$ is at most $2$ and, by \Cref{Stillness}, we get that $\Gal(\Q(\chi)/\Q)$ is also of order at most $2$  for every $\chi\in\Irr(G)$. Then $G$ is quadratic valued and, by \Cref{r-semirational}, $G$ is $r$-semi-rational.
\end{proof}

Next example shows that the hypothesis $r^2=1$ is necessary in \Cref{r2Semirat}.

\begin{example}{\rm 
The dihedral group $D_{32}$ of order $32$ is $\K_3$-still (for this group $\K_3=\Q(\sqrt{-2})$), but it is not USR because $G$ is not quadratic conjugated, as the conjugacy class of an element $g$ of order $16$ in $D_{32}$ does not contain $g^9$.
}
\end{example}
 
As an immediate corollary of \Cref{r2Semirat}, we obtain the following characterization of inverse semi-rational groups. See \cite{CD} and \cite{Bachle2018} for other characterizations.
  \begin{corollary}
      The group $G$ is inverse semi-rational if and only if $G$ is $\R$-still. 
  \end{corollary}

 Observe that the fields of the form $\K_r$ with $r\in \U_G$ are precisely the subfields of $\Q_G$ such that $\Q_G/\K_r$ is a cyclic extension. Hence \Cref{r-semirational} yields the following corollary.
	
\begin{corollary}\label{USRChar}
The following conditions are equivalent.
\begin{enumerate}
\item $G$ is USR. 
    \item $G$ is quadratic valued and $K$-still for some subfield $K$ of $\Q_G$ such that $\Q_G/K$ is a cyclic extension.
    \item $G$ is quadratic conjugated and $K$-still for some subfield $K$ of $\Q_G$ such that $\Q_G/K$ is a cyclic extension.     
\end{enumerate}
\end{corollary}

\Cref{USR} now follows from \Cref{Stillness} and \Cref{USRChar}.
Thought, not all the semi-rational groups are character quadratic, from \Cref{USRChar} we obtain the following corollary.
	
	\begin{corollary}\label{USRQR}
		If $G$ is USR, then it is character quadratic.   \end{corollary}
  \begin{proof}
      Suppose that $G$ is USR. By \Cref{USRChar},  $G$ is quadratic conjugated and there exists a subfield $K$ of $\Q_G$ such that $G$ is $K$-still and $\Q_G/K$ is a cyclic extension. Then for every $\chi\in \irr G$, $\Gal(\Q(\chi)/\Q)$ is elementary abelian $2$-group, by \Cref{CarQC}, and the restriction map $\Gal(\Q_G/K)\to \Gal(\Q(\chi)/\Q)$ is surjective, by \Cref{Stillness}. Therefore $\Gal(\Q(\chi)/\Q)$ is cyclic of order at most $2$, implying that $[\Q(\chi):\Q]\leq 2$. This proves that $G$ is character quadratic. 
  \end{proof}
	
\subsection{Quadratic groups and direct products}
    In this subsection we obtain two characterizations of quadratic groups. In the way, we characterize when the direct product of two groups is USR. Our first characterization follows by elementary arguments.

\begin{proposition}\label{QsRCar1}
The following are equivalent for a group $G$:
\begin{enumerate}
    \item\label{QGQuad} $G$ is quadratic. 
    \item\label{UG<=2} $[\U_G:\mcR_G]\le 2$.
    \item \label{SGQsR} $\mcS_G=\U_G$ or $\mcS_G=\U_G\setminus \mcR_G$.
    \item \label{ZQsR}$\zent {\Q G}\cong \Q^n\times K^m$ for an extension $K$ of $\Q$ of degree at most $2$ for some non-negative integer $n$ and a positive integer $m$. 
\end{enumerate}
In that case $K=\Q(G)$ and $G$ is USR. 
\end{proposition}

\begin{proof}
    The equivalence between \eqref{QGQuad} and \eqref{UG<=2} is a consequence of equation~\eqref{RGGal}. The fact that \eqref{SGQsR} implies \eqref{UG<=2} follows easily from \Cref{SRCoset}.
     Suppose that \eqref{UG<=2} holds. Then $[\U_G:\mcR_g]\leq 2$ for any $g\in G$ and, if the equality holds, then $\mcR_g=\mcR_G$ and $\mcS_g=\U_G\setminus \mcR_G$. Otherwise $\mcS_g=\U_G$. Therefore, either $\mcS_G=\mcR_G=\U_G$ or $\mcR_G\neq \U_G$ and $\mcS_G=\U_G\setminus \mcR_G$. Hence \eqref{UG<=2} implies \eqref{SGQsR}.
     The equivalence between \eqref{QGQuad} and \eqref{ZQsR} is a consequence of the fact that $\zent {\Q G}$ is isomorphic to the direct product of the fields of representatives of the orbits of irreducible (complex) characters of $G$ under the action of $\Gal(\Q_G/\Q)$ (see e.g. \cite[Theorem~3.3.1(4)]{JespersdelRioGRG1}).
\end{proof}
   
Let $G$ and $H$ be groups such that $\exp(G)$ is divisible by $\exp(H)$. Then $\pi_{G, H}$ denotes the natural map $\U_G\to \U_H$.  

\begin{remark}\label{RationalityQuotient}
If $H$ is an epimorphic image of $G$, then $\pi_{G,H}(\mcR_G)\subseteq \mcR_H$ and $\pi_{G,H}(\mcS_G)\subseteq \mcS_H$. In particular, if $G$ is USR, then so is $H$. 

Moreover the following diagram is commutative, where Res stands for the restriction map.
    $$\xymatrix{ \U_G \ar@{->}[r]^{\pi_{G,H}} \ar@{->}[d]_{\sigma_G} & \U_H \ar@{->}[d]^{\sigma_H} \\
    \Gal(\Q_G/\Q) \ar@{->}[r]^{{\rm Res}} & \Gal(\Q_H/\Q)
	}$$
Hence, if $K$ is a subfield of $\Q_H$, then
    \begin{equation}\label{CommDiagEq}
     \sigma^{-1}_G(\Gal(\Q_G/K))=\sigma^{-1}({\rm Res}^{-1}(\Gal(\Q_H/K)))=\pi^{-1}_{G,H}(\sigma_H^{-1}(\Gal(\Q_H/K))).
        \end{equation}
\end{remark}       
        
\begin{proposition}\label{DirectProductUSR}
Let $G$ and $H$ be groups. Then $G\times H$ is USR if and only if one of the following conditions hold:
\begin{enumerate}
	\item \label{RatTimesUSR} $G$ is rational and $H$ is USR, or viceversa. 
    \item \label{QsRTimes} $G$ and $H$ are quadratic and $\Q(G)=\Q(H)$.
\end{enumerate}
 In that case $\mcR_{G\times H}=\pi^{-1}_{G\times H,G}(\mcR_G)\cap \pi^{-1}_{G\times H,H}(\mcR_H)$ and $\mcS_{G\times H}=\pi^{-1}_{G\times H,G}(\mcS_G)\cap \pi^{-1}_{G\times H,H}(\mcS_H)$. Moreover, if \eqref{QsRTimes} holds, then $G\times H$ is quadratic and $\Q(G\times H)=\Q(G)$.
	\end{proposition}
	
\begin{proof}
We abbreviate $\pi_G=\pi_{G\times H,G}$ and $\pi_H=\pi_{G\times H,H}$.
		
First suppose that $G$ is rational and $H$ is USR. 
Then $\mcS_G=\U_G$ and we have to prove that  $\mcR_{G\times H}=\pi_H^{-1}(\mcR_H$) and $\mcS_{G\times H}=\pi_H^{-1}(\mcS_H$). 
By \Cref{RationalityQuotient}, $\mcS_{G\times H}\subseteq \pi_H^{-1}(\mcS_H)$.
To prove the reversal inclusion let $r\in \pi_H^{-1}(\mcS_H)$, $j\in \U_{G\times H}$, $g\in G$ and $h\in H$. As $g$ is rational $g,g^j$ and $g^r$ are conjugate in $G$. Moreover $h^j$ is conjugate to $h$ or $h^r$ in $H$. In the first case  $(g,h)^j$ is conjugate to $(g,h)$ in $G\times H$, in the second case it is conjugate to $(g,h)^r$ in $G\times H$.
Therefore $r\in \mcS_{G\times H}$. This proves that $\mcS_{G\times H}=\pi_H^{-1}(\mcS_H)$ and a similar argument shows that $\mcR_{G\times H}=\pi_H^{-1}(\mcR_H)$. 
		
Now suppose that $G$ and $H$ are quadratic and $\Q(G)=\Q(H)=K$.  
By the previous case, we may assume that $K\neq \Q$, and we have to show that $G\times H$ is quadratic, $\mcR_{G\times H}=\pi^{-1}_G(\mcR_G)=\pi_H^{-1}(\mcR_H)$ and  $\mcS_{G\times H}=\pi^{-1}_G(\mcS_G)=\pi^{-1}_H(\mcS_H)$.   
        By \Cref{QsRCar1}, $\zent {\Q G}$ is isomorphic to $\Q^{n_1}\times K^{m_1}$ and $\zent{\Q H}$ is isomorphic to $\Q^{n_2}\times K^{m_2}$ for some non-negative integers $n_1, n_2, m_1$ and $m_2$. 
        Then $\zent{\Q (G\times H) }\cong \zent{\Q G}\otimes_\Q \zent{\Q H} \cong \Q^n \times K^m$ where $n=n_1n_2$ and $m=n_1m_2+n_2m_1+2m_1m_2$.
        So, by \Cref{QsRCar1}, the group $G\times H$ is quadratic and $K=\Q(G\times H)$. By equations~\eqref{RGGal} and \eqref{CommDiagEq}, we have that $\mcR_{G\times H}=\sigma_{G\times H}^{-1}(\Gal(\Q_{G\times H}/K))=\pi^{-1}_G(\sigma_G^{-1}(\Gal(\Q_G/K)))=\pi^{-1}_G(\mcR_G)$. Moreover, by \Cref{QsRCar1}, $$\mcS_{G\times H}=\U_{G\times H} \setminus \mcR_{G\times H}=\U_{G\times H} \setminus \pi^{-1}_G(\mcR_G)=\pi^{-1}_G(\U_G\setminus \mcR_G)=\pi^{-1}_G(\mcS_G).$$ 
        By symmetry, $\mcR_{G\times H}=\pi_H^{-1}(\mcR_H)$ and  $\mcS_{G\times H}=\pi^{-1}_H(\mcS_H)$.  This proves the sufficient part of the proposition and the last two statements.
        
		To prove the necessary part suppose that $G\times H$ is USR. Then, by \Cref{RationalityQuotient}, both $G$ and $H$ are USR. If one of them is rational, then condition \eqref{RatTimesUSR}  holds. 
  If both are rational we are also in case \eqref{QsRTimes} and $G\times H$ is rational. 
Suppose that none of them is rational. If $\chi\in \Irr(G)$ and $\phi\in \Irr(H)$, then $\chi\times \psi\in \Irr(G\times H)$ and hence $[\Q(\chi\times \psi):\Q]\leq 2$. This implies that either $\Q(\chi)\subseteq \Q(\psi)$ or $\Q(\psi)\subseteq \Q(\chi)$.
		On the other hand since $G$ is not rational, there is a $\chi_0\in \Irr(G)$ such $\Q(\chi_0)\ne \Q$ and this implies that $\Q(\psi)\subseteq \Q(\chi_0)$ for every $\psi\in \Irr(H)$. As $H$ is also not rational we conclude that there is a quadratic field $K$ containing $\Q(\chi)$ and $\Q(\psi)$ for every $\chi\in\Irr(G)$ and every $\psi\in\Irr(H)$. Then $\Q(G)=\Q(H)=\Q(G\times H)=K$.
	\end{proof}
	
    \begin{corollary}
		Let $G$ be a USR group which is not quadratic. If $G=G_1\times G_2$, then either $G_1$ or $G_2$ is rational.
	\end{corollary}

We conclude the section with the second characterization of quadratic groups.
 
\begin{corollary}\label{PowerUSR}
The following are equivalent for a group $G$:
\begin{enumerate}
    \item\label{QsRPowerUSR} $G$ is quadratic. 
    \item\label{PQsR} $G^n$ is quadratic for every $n\ge 1$.
    \item\label{P2USR} $G\times G$ is USR.
	\item\label{PUSR} $G^n$ is USR for every $n\ge 1$.
\end{enumerate}
\end{corollary}

	\section{The rationality and semi-rationality of nilpotent groups}\label{SecNilpotent}
	
In this section we compute the subsets of $\U_n$ that occur as rationality or semi-rationality of a nilpotent group of exponent $n$. In particular, we classify the fields that occur as $\Q(G)$ for a nilpotent group $G$. To that aim the following description of the irreducible characters of a restricted wreath product, with an acting group of prime order, turns out to be a useful tool.

\begin{theorem}\cite{Revin2004}\label{Revin}
Let $X$ be a finite group, $p$ a prime number and $G=X\wr C_p$, the wreath product of $X$ and a cyclic group $\GEN{t}$ of order $p$. 
Then $\Irr(G)$ is formed by the following characters: 
\begin{itemize}
\item First type: $(\chi_1\cdots\chi_p)^G$, where $\{\chi_1,\dots,\chi_p\}\subseteq \Irr(X)$. Moreover, if  $\{\varphi_1,\dots,\varphi_p\}\subseteq \Irr(X)$, then $(\chi_1\cdots\chi_p)^G=(\varphi_1, \dots, \varphi_p)^G$ if and only if $(\varphi_1,\dots,\varphi_p)$ is a cyclic permutation of $(\chi_1,\dots,\chi_p)$.
\item Second type: $\chi^j$, where $\chi\in\Irr(X)$ and $j\in\{0,1,\dots,p-1\}$, defined as follows:
    $$\chi^j(t^i(x_0,x_1,\dots ,x_{p-1})) = \begin{cases}
        \chi(x_0)\chi(x_1)\cdots \chi(x_{p-1}), & \text{if } i=0; \\ \zeta_p^{ij}\chi(x_0x_ix_{2i}\cdots x_{(p-1)i}), & \text{if }i\ne 0.
    \end{cases}$$
where the multiplication in the subindices of the $x_i$'s is performed modulo $p$. 
\end{itemize}
\end{theorem}

\begin{corollary}\label{WreathCp}
If $X$ is a group and $G=X\wr C_p$, then $\exp(G)=p\exp(X)$ and $\Q(G)=\Q(X)(\zeta_p)$.
\end{corollary}

\subsection{Rationality}

The following theorem classifies the fields that occur as $\Q(G)$ for a nilpotent group $G$ of exponent $n$. Its corollary describes the subgroups of $\U_n$ that occur as the rationality of a nilpotent group of exponent $n$. 
Note that \Cref{QGNilp} and \Cref{RatNilp} contain the first two statements of \Cref{Nilpotent}.

\begin{theorem}\label{QGNilp}
Let $n$ be a positive integer and let $K$ be a subfield of $\C$. Let $n'$ denote the greatest square-free divisor of $n$. Then the following conditions are equivalent: 
\begin{enumerate}
    \item\label{QGNilp1} $K=\Q(G)$ for a nilpotent group $G$ of exponent $n$.
    \item\label{QGNiln'} $\Q_{n'}\subseteq K\subseteq \Q_n$.
    \item\label{QGNilK} $K=F(\zeta_m)$ for $m$ a divisor of $n$ divisible by  $n'$ and 
    $$F\in \{\Q(\zeta_{2^k}+\zeta_{2^{k}}^{-1}),\Q(\zeta_{2^k}+\zeta_{2^k}^{-1+2^{k-1}}) : 2^k\mid n\}.$$
\end{enumerate}
\end{theorem}

\begin{proof}
For every positive integer $k$ and a prime $p$, let $k_p$ denote the greatest power of $k$ dividing $k$. Moreover, set
    $$\F_{k,1} = \Q(\zeta_{2^k}+\zeta_{2^{k}}^{-1}) \qand \F_{k,2}=\Q(\zeta_{2^k}+\zeta_{2^k}^{-1+2^{k-1}}).$$
Observe that $\F_{1,1}=\F_{2,1}=\F_{1,2}=\Q$ and $\F_{2,2}=\Q_4$.

We begin by proving \eqref{QGNilp1} implies \eqref{QGNiln'}.
Suppose first that $K=\Q(G)$ with $G$ a $p$-group of exponent $n$. 
In particular, $K$ is a subfield of $\Q_n$ and $\mcR_G$ is a $p$-subgroup of $\U_G$, hence so is $\Gal(\Q_G/K)$, by equation~\eqref{RGGal}. 
As the Sylow $p$-subgroup of $\Gal(\Q_n/\Q)$ is $\Gal(\Q_n/\Q_p)$, it follows that $\Q_p\subseteq K\subseteq \Q_n$.

Now suppose that $K=\Q(G)$ for an arbitrary nilpotent group $G$ of exponent $n$. 
For every $p\in \pi(n)$, let $G_p$ denote the Sylow $p$-subgroup of $G$. So $G=\prod_{p\in \pi(n)} G_p$ and hence $K$ is the compositum of the fields of the form $\Q(G_p)$ with $p\in \pi(n)$. By the previous paragraph, $\Q_p\subseteq \Q(G_p)\subseteq \Q_{n_p}$ for every $p\in \pi(n)$. 
Therefore $\Q_{n'}\subseteq K\subseteq \Q_n$.

Secondly, we prove that \eqref{QGNiln'} implies \eqref{QGNilK}.
Let $K$ be a subfield of $\Q_n$ containing $\Q_{n'}$.
Then $\Gal(\Q_n/K)$ is a subgroup of $\Gal(\Q_n/\Q_{n'}) \cong \prod_{p\in \pi(n)} \Gal(\Q_{n_p}/\Q_p)$. Moreover $\Gal(\Q_{n_p}/\Q_p)$ is a $p$-group. Therefore the image of $\Gal(\Q_n/K)$ under the previous isomorphism is a direct product of groups of the form $\Gal(\Q_{n_p}/K_p)$ where $K_p$ is a subfield of $\Q_{n_p}$ containing $\Q_p$. If $p$ is odd then $K_p=\Q_{m_p}$ for $m_p$ a divisor of $n_p$ other than $1$. 
Moreover, every subfield of $\Q_{n_2}$ is of one of the following forms: $\Q(\zeta_{2^k})$, $\F_{k,1}$ or $\F_{k,2}$ with $2^k\mid n_2$. In particular, so is $K_2$. Therefore $K$ is the compositum of the fields $K_p$ with $p\in \pi(n)$. Then $K=K_2(\zeta_{m_{2'}})=F(\zeta_m)$, where $m_{2'}=\prod_{p\in \pi(n)\setminus \{2\}}{m_p}$ and $F=\F_{k,i}$, for $2^k\mid n$ and $i\in \{1,2\}$.

Finally, we prove that \eqref{QGNilK} implies \eqref{QGNilp1}.
Suppose that $K=F(\zeta_m)$, where $F=\F_{k,i}$ with $m$ and $2^k$ divisors of $n$ and $i\in \{1,2\}$.
For every $p\in \pi(n)$ let 
    $$K_p=\begin{cases} F(\zeta_{m_2}), & \text{if } p=2; \\
    \Q_{m_p}, & \text{otherwise}.\end{cases}$$
Then $K$ is the compositum of the fields $K_p$ with $p\in \pi(n)$. Therefore, if $G_p$ is a $p$-group of exponent $n_p$ with $\Q(G_p)=K_p$ for each $p\in \pi(n)$, then the group $G=\prod_{p\in \pi(n)} G_p$ is nilpotent of exponent $n$ with $K=\Q(G)$. This reduces the proof to the case where $n=p^e$ for some prime $p$ and a positive integer $e$. Then $m=p^t$ for some $1\le t\le e$. If $p$ is odd then $K=\Q_{p^t}$, and, by \Cref{WreathCp}, the group $G=C_{p^{t}}\underbrace{\wr C_p\wr \dots \wr C_p}_{e-t \text{ times}}$ has exponent $p^e$ and $\Q(G)=K$, as desired.

Suppose that $p=2$. Then
    $$K=\F_{k,i}(\zeta_{2^t}) = \begin{cases} \Q_{2^{\max(k,t)}}, & \text{if } t\ge 2, \\ \F_{k,i}, & \text{otherwise}.\end{cases}$$
Therefore, either $K=\Q_{2^u}$ with $u=\max(k,t)\le e$, or $K=\F_{k,i}$ with $3\le k \le e$ and $i\in \{1,2\}$.
Hence it remains to show that for each of these fields $K$ there is a group $G$ of exponent $2^e$ with $\Q(G)=K$. 
To achieve this we observe that 
    $$\Q(C_{2^u})=\Q_{2^u}, \quad \Q(D_{2^{k+1}})=\F_{k,1} \qand
    \Q(D^-_{2^{k+1}})=\F_{k,2}, \text{ for } k\ge 3.$$
To increase the exponent we consider the groups
    $$G_{u,e}=C_{2^{u}}\underbrace{\wr C_2\wr \dots \wr C_2}_{e-u \text{ times}}, \quad 
D_{k,e}=D_{2^{k+1}}\underbrace{\wr C_2\wr \dots \wr C_2}_{e-k \text{ times}} \qand 
D^-_{k,e}=D^-_{2^{k+1}}\underbrace{\wr C_2\wr \dots \wr C_2}_{e-k \text{ times}}.$$
By \Cref{WreathCp}, these groups have exponent $2^e$, and
    $$\Q(G_{u,e})=\Q_{2^u}, \quad \Q(D_{k,e})=\F_{k,1} \qand \Q(D^-_{k,e})=\F_{k,2}.$$
This finishes the proof.
\end{proof}

Combining \Cref{QGNilp} and equation~\eqref{RGGal}, we obtain the following corollary.

\begin{corollary}\label{RatNilp}
Let $n$ be positive integer and let $R$ be a subset of $\U_n$. Then $R$ is the rationality of a nilpotent group of exponent $n$ if and only if 
$R$ is a subgroup of $\{r\in \U_n : r\equiv 1 \mod n'\}$, where $n'$ is the greatest square-free divisor of $n$.
\end{corollary}

\begin{corollary}\label{USRNilpPrimes}
If $G$ is a nilpotent $USR$ group,  then $\pi(G)\subseteq \{2,3\}$.    
\end{corollary}
\begin{proof}
Let $G$ be an USR nilpotent group and let $p\in \pi(G)$. Thus $G$ is quadratic conjugated by \Cref{USRChar}, and hence $\Gal(\Q(G)/\Q)$ is an elementary abelian $2$-group, by \Cref{CarQC}. Moreover, $\Q_p\subseteq \Q(G)$, by \Cref{QGNilp} and hence $\Gal(\Q_p/\Q)$ is also elementary abelian $2$-group. This implies that $p\in \{2,3\}$.
\end{proof}

\subsection{Semi-rationality}
In this subsection we classify, for every positive integer $n$, the cosets of $\U_n$ that occur as the semi-rationality of an USR nilpotent group of exponent $n$. 
We begin with a result which shows that, while \Cref{NoConverses} demonstrate that no converse of the implications in \Cref{Implications} holds for $2$-groups, for $3$-groups all the properties displayed, except that of being rational, are equivalent.

\begin{proposition}\label{USR3-Group}
The following are equivalent for a $3$-group $G$.
\begin{enumerate}
\item $G$ is inverse semi-rational.
\item $G$ is quadratic.
\item $G$ is USR.
\item $G$ is semi-rational. 
\item $G$ is character quadratic.
\item $G$ is quadratic valued.
\item $G$ is quadratic conjugated.
\item $4\in \mcR_G$.
\end{enumerate}
In this cases $\mcR_G=\U_G^2=\GEN{4}$, $\mcS_G=-\GEN{4}$ and if $G\ne 1$, then $\Q(G)=\Q(\zeta_3)$. 
\end{proposition}
	
\begin{proof} 
(1) implies (3), (3) implies (4), (4) implies (6),  and (5) implies (6) follow from the definitions. 

(2) implies (3) follows from \Cref{QsRCar1}; (3) implies (5) follows from \Cref{USRQR}; and (6) implies (7) follows from \Cref{QVImpliesQC}. 

(7) implies (8) is obvious. 

(8) implies (1) and (2). 
Suppose that $4\in \mcR_G$. If $G=1$, then it is rational and hence it is both inverse semi-rational and quadratic. Otherwise $-1\not\in\mcR_G$  and, as $[\U_G:\GEN{4}]=2$, it follows that $\mcR_G=\GEN{4}$. Thus $G$ is quadratic and $\mcS_G=\U_G\setminus \GEN{4}$, by \Cref{QsRCar1}. Hence $-1\in\mcS_G$, i.e. $G$ is inverse semi-rational. 
Moreover, by equation~\eqref{RGGal}, $\Gal(\Q_G/\Q(G))=\GEN{\sigma_4}=\Gal(\Q_G/\Q(\zeta_3))$ and hence $\Q(G)=\Q(\zeta_3)$. This concludes the proof.
	\end{proof}

\begin{proposition}\label{USRWreath}
Let $X$ be a group, $p$ a prime integer and $G=X\wr C_p$.
Then the following conditions are equivalent:
\begin{enumerate}
\item  \label{WQsR} $G$ is quadratic.
\item \label{WUSR} $G$ is USR.
\item One of the following conditions hold:
\begin{enumerate}\label{Wp}
\item \label{Wp2} $p=2$ and $X$ is quadratic;
\item \label{Wp3} $p=3$ and $\Q(X)\subseteq \Q(\zeta_3)$.
\end{enumerate}
\end{enumerate}
Furthermore, in case \eqref{Wp2} $\Q(G)=\Q(X)$, and in case \eqref{Wp3} $\Q(G)=\Q(\zeta_3)$.
\end{proposition}

\begin{proof}
\eqref{WQsR} implies \eqref{WUSR} follows from \Cref{QsRCar1}. 

\eqref{Wp} implies \eqref{WQsR} follows from \Cref{WreathCp}.

Now we prove that \eqref{WUSR} implies \eqref{Wp}. 
Suppose that $G$ is USR. Then $C_p$ is USR, by \Cref{RationalityQuotient}, and hence $p$ is $2$ or $3$.
By means of contradiction, suppose that $X$ is not quadratic. Then there exist $\chi_1,\chi_2\in \Irr(X)$ such that $\Q(\chi_1)$ and $\Q(\chi_2)$ are different quadratic extensions of $\Q$, so that there are $x_1,x_2\in X$ with $\chi_2(x_1)\not\in \Q(\chi_1)$ and $\chi_1(x_2)\not\in \Q(\chi_2)$. 
Consider the character $\chi=(\chi_1\chi_2)^G$ of $G$. 
Observe that this is a character of the first type in \Cref{Revin} because $\chi=(\chi_1\chi_2\chi_0\dots \chi_0)^G$ for $\chi_0$ the trivial character of $X$.  
Then $\chi\in \Irr(G)$. Then $\chi(x_2,1,1,\dots,1)=\chi_1(x_2)$ and $\chi(1,x_1,1,\dots,1)=\chi_2(x_1)$. Therefore $\Q(\chi_1(x_2),\chi_2(x_1))\subseteq \Q(\chi)$ and hence $[\Q(\chi):\Q]\geq [\Q(\chi_1(x_2),\chi_2(x_1)):\Q]=4$. Therefore, $G$ is not character quadratic, in contradiction with \Cref{USRQR}.
Thus $X$ is quadratic.
If $p=2$, then condition \eqref{Wp2} holds.
Now suppose that $p=3$. Then for any irreducible character $\chi^1$ of $G$ of the second type in \Cref{Revin}, we get $\chi^1(t)=\zeta_3$ so that $\Q(\chi^1)=\Q(\zeta_3)$. Then $\chi(x)=\chi^1(x,1,\dots,1)\in \Q(\zeta_3)$ for every $x\in X$. This shows that $\Q(X)\subseteq \Q(\zeta_3)$, i.e. condition \eqref{Wp3} holds. 

The last statement is a consequence of \Cref{WreathCp}.
\end{proof}

We conclude with the following theorem which is the third statement of \Cref{Nilpotent}.

\begin{theorem}\label{USRNilpotent}
Let $n$ be a positive integer and let $S$ be a subset of $\U_n$. Then $S$ is the semi-rationality of a nilpotent USR group of exponent $n$ if and only if $S$ is an admissible coset of $\U_n$, 
$\pi(n)\subseteq \{2,3\}$ and, if $3\in \pi(n)$, then $S=\{x\in \U_n : x\equiv -1 \mod 3\}$.
\end{theorem}
	
\begin{proof}
Suppose that $S=\mcS_G$ for a nilpotent group $G$ of exponent $n$. 
Then $S$ is an admissible coset of $\U_n$, by \Cref{SRCoset}.
By \Cref{USRNilpPrimes}, $\pi(n)=\pi(G)\subseteq \{2,3\}$.
Thus $G=G_2\times G_3$, with $G_p$ the Sylow $p$-subgroup of $G$.
Moreover, $G_2$ and $G_3$ are USR.
Suppose that $n$ is divisible by $3$. Then, by \Cref{USR3-Group}, 
    $$\mcS_{G_3}=-1\GEN{4}=\{r\in \U_{G_3}:\; r\equiv -1 \mod 3  \}$$ 
and $\Q(G_3)=\Q(\zeta_3)$. In particular, $G_3$ is not rational.
Since $\Q(G_2)\subseteq \Q_{G_2}$ and $\Q_{G_2}\cap \Q_{G_3}=\Q$, we get $\Q(G_2)\neq \Q(G_3)$. 
Then $G_2$ is rational and $S=\mcS_G=\pi_{G,G_2}^{-1}(\mcS_{G_2})\cap \pi_{G,G_3}^{-1}(\mcS_{G_3})=\{r\in \U_n:\; r\equiv -1\mod 3\}$, by \Cref{DirectProductUSR}. 
This proves the necessary condition. 
		
Conversely, suppose that $S$ is an admissible coset of $\U_n$, $\pi(n)\subseteq \{2,3\}$ and if $3\in \pi(n)$ then $S=\{x\in \U_n : x\equiv -1 \mod 3\}$.
Consider the groups $X_k$ and $Y_k$ defined recursively by setting 
    $$X_1=C_2, \quad X_{k+1}=X_k\wr C_2, \quad Y_1=C_3, \quad Y_{k+1}=Y_k\wr C_3.$$
By \Cref{WreathCp}, $X_k$ is rational of exponent $2^k$ and $Y_k$ is quadratic, in particular USR, with $\mcS_{Y_{k}}=-\GEN{4}=\{r\in \U_{3^k} : r\equiv -1 \mod 3\}$.
So we may assume that either $\pi(n)=\{2\}$ and $S\ne \U_n$, or $\pi(n)=\{2,3\}$ and $S=\{r\in \U_n : r\equiv -1 \mod 3\}$. 

Suppose firstly that $\pi(n)=\{2,3\}$ and $S=\{x\in \U_n : x\equiv -1 \mod 3\}$. Write $n=2^k3^m$ and let $G=X_k\times Y_m$.
By \Cref{DirectProductUSR} and \Cref{USR3-Group}, $G$ is an USR group of exponent $n$ and $\mcS_G=\pi_{G,X_k}^{-1}(\U_{X_k})\cap \pi_{G,Y_m}^{-1}(-\GEN{4})=S$. 

Finally, suppose that $\pi(n)=\{2\}$ and $S\ne \U_n$. Thus $n$ is multiple of $4$. If $n=4$, then $S=\{-1\}$ and $\mcS_{C_4}=S$. 
So in the remainder of this case we assume that $n=2^k$ with $k\ge 3$. 
Consider the following groups:		
\begin{eqnarray*}
H_1 &=& \GEN{a}_8\rtimes \GEN{x}_2, a^x=a^5;\\
H_2 &=& \GEN{a}_8\rtimes \GEN{x}_2, a^x=a^3; \\ 
H_3 &=&\GEN{a}_8\rtimes \GEN{x}_2, a^x=a^{-1};\\
H_4 &=&  \GEN{a}_8\rtimes \GEN{x}_4, a^x=a^3; \\ 
H_5 &=&  \GEN{a}_8\rtimes \GEN{x}_4, a^x=a^{-1}; \\ 
H_6 &=& (\GEN{a}_8\times \GEN{b}_4)\rtimes (\GEN{x}_2\times \GEN{y}_2), a^x=a^{-1},b^x=a^4b^{-1},a^y=a^3b^2,b^y=b^{-1}.
\end{eqnarray*}
Their identification in the \texttt{GAP} library of small groups are  \texttt{[16,6]}, \texttt{[16,8]}, \texttt{[16,7]}, \texttt{[32,13]}, \texttt{[32,14]} and \texttt{[128,1956]}, respectively.
A straightforward calculation shows that each $H_i$ has exponent 8 and their semi-rationalities are
$$\begin{array}{cccccc}
\mcS_{H_1}=  -\GEN{5}, & \mcS_{H_2}=-\GEN{-5}, & \mcS_{H_3}= 5\GEN{-1}, &\mcS_{H_4}= \{-1\}, & \mcS_{H_5}= \{-5\}, & \mcS_{H_6}= \{5\}.
\end{array}$$
Let $G_{k,i}=X_k\times H_i$. By \Cref{DirectProductUSR}, $\mcS_{G_{k,i}}=\pi^{-1}_{G_{k,i},H_i}(\mcS_{H_i})$ and using the description of $\mcS_{H_i}$ it follows that 

$$\begin{array}{ccc}
\mcS_{G_{k,1}}=  -\GEN{5}, & \mcS_{G_{k,2}}=-\GEN{-5}, & \mcS_{G_{k,3}}= 5\GEN{5^2,-1}, \\
\mcS_{G_{k,4}}= -\GEN{5^2}, & \mcS_{G_{k,5}}= -5\GEN{5^2}, & \mcS_{G_{k,6}}= 5\GEN{5^2}
\end{array}$$
As these are all the admissible cosets of $\U_{2^k}$ other than $\U_{2^k}$, this completes the proof for this case.
\end{proof}

\textbf{Acknowledgment}. The authors thank the reviewer for the careful reading and insightful comments.

\printbibliography
\end{document}